\documentclass[12pt,english]{amsart}
\usepackage{amssymb,amsmath,amscd,graphicx,fontenc,bbold,amsthm,mathrsfs,mathtools}
\usepackage[pdftex,bookmarks,colorlinks,breaklinks]{hyperref} 
 \usepackage{comment}
\hypersetup{linkcolor=blue,citecolor=red,filecolor=dullmagenta,urlcolor=blue}
\newtheorem{theorem}{Theorem}[section]

\newtheorem{lemma}[theorem]{Lemma}
\newtheorem{proposition}[theorem]{Proposition}
\newtheorem{corollary}[theorem]{Corollary}
\theoremstyle{definition}

\newtheorem{conj}{Conjecture}

\usepackage[top=3cm,bottom=3cm,right=2.3cm,left=2.3cm,twoside=false]{geometry}
\title{the Multiplication table for Smooth integers }
\author{Marzieh Mehdizadeh}
\address{ D\'{e}partment de Math\'{e}matiques et Statistique,Universit\'{e} de Montr\'{e}al, CP 6128, succ.
Centre-ville, Montr\'{e}al, QC, Canada H3C 3J7.}
\email{marzieh.mehdizadeh@gmail.com}
\date{}
\begin{document}

\maketitle

\begin{abstract}
The Erd\H{o}s multiplication table problem asks what is the number of distinct integers appearing in the $N\times N$ multiplication table. The order of magnitude of this quantity was determined by Ford ~\cite{y2y}. In this paper we study the number of $y-$smooth entries of the $N\times N$ multiplication table that is to say entries with no prime factors greater than $y$.
\end{abstract}

\section{Introduction}
The \textit{multiplication table problem} involves estimating 
$$
A(x):=\#\{ab: a,b\leq \sqrt{x}, \quad \text{and} \quad a,b \in \mathbb{N}\}.
$$  
This interesting question, posed by Erd\H{o}s, has been studied by many authors. Erd\H{o}s in ~ \cite{Erdos55}, showed that for all $\varepsilon>0$, we have
\begin{equation}\label{erd}
\frac{x}{(\log x)^{\delta+\epsilon}}\leq A(x)\leq \frac{x}{(\log x)^{\delta-\varepsilon}} \quad (x\to \infty),
\end{equation}
where
\begin{equation}\label{delta}
\delta=1-\frac{1+\log\log 2}{\log 2}= 0.0860\dots .
\end{equation}
The best estimate of $A(x)$ is a result due to Kevin Ford ~\cite{y2y}. He proved the following estimate, that significantly improved the order of magnitude of $A(x)$ as follows
\begin{equation}\label{Ford-08}
A(x)\asymp \frac{x}{(\log x)^{\delta}(\log \log x)^{3/2}}.
\end{equation}\\\\
\textbf{Notation:} In this paper, we use the notation  $f(x) \asymp g(x)$ if both $f(x)\ll g(x)$ and $g(x)\ll f(x)$ hold, where $f(x)\ll g(x)$ or $f(x)=O(g(x))$ interchangeably to mean that $|f(x)| \leq cg(x) $ holds with some constant $c$ for all $x$ in a range which will normally be clear from the context. Also, the notation $f(x)\sim g(x)$ means that $f(x)/g(x)\to 1$ as $x\to \infty$, and $f(x)=o(g(x))$ means that $f(x)/g(x)\to 0$ as $x\to \infty$.\\
Also, $u$ is defined as
\begin{equation*}
u:=\frac{\log x}{\log y} \qquad x\geq y\geq 2,
\end{equation*}
and we let $\log_k x$ denote the $k$-fold iterated logarithm, defined by $\log_1x:= \log x$ and $\log_k x= \log \log_{k-1}x$, for $k>1$.\\\\
Motivated by this background, in this paper we investigate the multiplication table problem for smooth integers. The set of $y-$smooth numbers, is defined by
$$S(x,y):= \{n\leq x : P(n)\leq y\},$$
where $P(n)$ denotes the largest prime factor of an integer $n \geq 2$, with the convention $P(1)=1$. Set
$$\Psi(x,y):= \big|S(x,y)\big |.$$
Our main aim in this work is to study
$$A(x,y):= \# \{ab: a,b\in S(\sqrt{x},y) \}.$$
Hence computing $A(x,y)$ is equivalent to estimating the size of  $S(\sqrt{x},y)\cdot S(\sqrt{x},y)$.\\
A simple approximation of $\Psi(x,y)$ proved by Canfield, Erd\H{o}s and Pomerance ~\cite{can} states that for a fixed $\epsilon>0$, we have
\begin{equation}\label{simple}
\Psi(x,y)= x u^{-u(1+o(1))} \quad \text{as}\quad  u\to \infty,
\end{equation}
for $u\leq y^{1-\epsilon}$, that is $y\geq (\log x)^{1+\epsilon}$.\\
By estimate \eqref{simple}, one can see that for $u$ large (or $y$ small), the value of $\Psi(x,y)$ is small. It counts the integers having large number of prime factors. Since in this case every $n$ has a lot of small prime factors, we can find $a$ and $b$ such that $n=ab$ and $a,b\leq \sqrt{x}$.\\\\
If $u$ is small (which means that $y$ is large), then by \eqref{simple}, one can deduce that the value of $\Psi(x,y)$ is large compared to $x$. In this case, $S(x,y)$ contains integers with large prime factors and we expect the size of $S(\sqrt{x},y)\cdot S(\sqrt{x},y)$ to be small.\\\\
It is good to mention that by a connection to sum-product problem, Banks and Covert ~\cite{banks} by invoking combinatorial tools,  have considered the behaviour of $A(x^2,y)= \big|S(x,y)\cdot S(x,y)\big|$ in different ranges of $y$, particularly for the cases when $y$ is relatively small or large. 

Here we present a simple idea to prove that $A(x,y)$ has a same size as $\Psi(x,y)$ when $y$ is small compared to $\log x$. Let  $n\leq \frac{x}{y}$ be a $y-$smooth number. If $n \leq \sqrt{x}$ then trivially we have $n\in A(x,y)$. Thus, we assume that $\sqrt{x} \leq n$.  Let $p_1\leq p_2\leq \cdots\leq p_k$ be prime factors of $n$. Consider the following sequence obtained by prime factors of $n$:  
$$n_0=1,\qquad n_j= \prod _{i=1}^{j} p_i, \qquad 1\leq j \leq k.$$
 Since $n\geq \sqrt{x}$ then there exists a unique integer $s$, with $0\leq s <k$ such that $n_s< \sqrt{x}\leq n_{s+1}$. Each prime factor of $n$ is less than $y$, therefore
 $$n_s \leq \sqrt{x} \leq n_{s+1}\leq n_s y.$$
 Set  $d=n_s$, then
 $$
\frac{\sqrt{x}}{y}\leq d\leq \sqrt{x}. 
 $$
 Since $n\leq x/y$, then we easily conclude that 
 $$
\frac{n}{d}\leq \sqrt{x}.
 $$
 Therefore,
 $$\Psi(x/y,y) \leq A(x,y) \leq \Psi(x,y),$$
and by a simple argument one can deduce that as $x,y \to \infty$ then $\Psi(x/y,y) \sim \Psi(x,y)$ when $y=o(\log x)$, (see Lemma \ref{smally}).
This argument leads us to state the following theorem.\\
 \begin{theorem}\label{thm3}
 If $y=o(\log x)$ then we have
 $$A(x,y) \sim \Psi(x,y) \quad \text{as} \quad x,y\rightarrow \infty.$$
 \end{theorem}
 \bigskip
The problem gets harder, and hence, more interesting when $y$ takes larger values compared to $\log x$. We shall prove the following theorem for small values of $y$ compared to $x$.
\begin{theorem} \label{thm1} We have
$$A(x,y)\sim \Psi(x,y) \quad \text{as}\quad x,y\to \infty,$$
when $u$ and $y$ satisfy the range
\begin{equation}\label{range1}
\frac{u\log u}{(\log y \log_2 y \log_3 y)^{2}} \to \infty ,\quad \text{which implies}, \quad y\leq \exp \left\lbrace\frac{(\log x)^{1/3}}{(\log_2 x)^{1/3 +\epsilon}}\right\rbrace,
\end{equation}
for $\epsilon>0$ arbitrarily small.
\end{theorem}
 Theorem~\ref{thm1} is proved in Section $3$. The proof relies on some probabilistic arguments and recent estimates for $\Psi(x/p,y)$ where $p$ is a prime factor of $n$.\\\\

If $y$ takes values very close to $x$, which implies $u$ is small compared to $\log \log y$, then we will show the following theorem.
 \begin{theorem}\label{thm2}
Let $\epsilon>0$ is arbitrarily small, then we have
$$A(x,y)= o(\Psi(x,y)) \quad \text{as} \quad x,y\to \infty,$$
where $u$ and $y$ satisfying the range
\begin{equation} \label{range2}
u < (L-\epsilon)\log_2 y, \quad \text{which implies}, \quad y \geq \exp \left\lbrace\frac{\log x}{(L-\epsilon)\log_2 x} \right\rbrace,
\end{equation}
where $L:= \frac{1-\log 2}{\log 2}$.
\end{theorem}  
Theorem \ref{thm2} is proved in Section $4$, by applying an Erd\H{o}s' idea~\cite{erdos}, suitably modified for $y-$smooth integers.\\ \\
In what follows, we will give a heuristic argument that predicts the behaviour of $A(x,y)$ in ranges \eqref{range1} and \eqref{range2}.\\
We define the function $\tau(n; A,B)$ to be the number of all divisors of $n$ in the interval $(A,B]$. In other words.
$$\tau(n; A,B) :=\# \{d: d|n \Rightarrow A< d \leq B\}. $$
Let $n\in S((1-\eta)x,y)$ be a square-free number with $k$ prime factors, where $\eta \to 0$ as $x\to \infty$. Assume that the set 
$$D(n):=\{\log d: d|n\}$$ 
is uniformly distributed in the interval $[0, \log n]$. So 
\begin{equation}
P\left(d\in (A,B)\right):=\tau(n)  \frac{\log B- \log A}{\log n},
\end{equation}
where the sample space is defined by
$$S:= \left\lbrace n\leq x: \omega(n)= k \right\rbrace,$$
and $n$ being chosen uniformly at random. By this assumption, the expected value of the function $\tau(n,(1-\eta)\sqrt{x},\sqrt{x})$ is as follows
\begin{equation}\label{Exp1}
\mathbb{E} \left[\tau(n,(1-\eta)\sqrt{x},\sqrt{x})\right]=
\frac{2^k \log (1/(1-\eta))}{\log \sqrt{x}} \asymp \frac{2^{k}}{u\log y}.
\end{equation}
 Alladi and Hildebrand in ~\cite{Alladi} and ~\cite{Hild} showed that the normal number of prime factors of $y-$smooth integers is very close to its expected value $u+\log_2 y$ in different ranges of $y$. Hence, from~\eqref{Exp1}, we deduce that
 $$\mathbb{E} \left[\tau(n,(1-\eta)\sqrt{x},\sqrt{x})\right] \asymp \frac{2^{u+\log_2 y}}{\log y}.$$
If $2^{u+\log_2 y}/\log y\to \infty$, then we expect that $n$ will have a divisor $d$ in the interval $((1-\eta)\sqrt{x},\sqrt{x}]$.\\ We know $n\leq (1-\eta) x$. Thus, $n/d \leq \sqrt{x}$, and we can deduce that $n\in A(x,y)$, this means that 
$$\Psi((1-\eta)x,y) \leq A(x,y).$$
Trivially $A(x,y)\leq \Psi(x,y)$. So by  this argument, we obtain 
$$
A(x,y)\sim \Psi(x,y),
$$
when $\eta \to 0$ as $x\to \infty$. \\\\
On the other hand, if  $2^{u+\log_2 y}/\log y \to 0$, then we expect that none of integers in  $S((1-\eta)x,y)$ have a divisor in $((1-\eta)\sqrt{x}, \sqrt{x}]$ (except a set with density $0$), this means that 
$$A(x,y)=o(\Psi(x,y)) \quad \text{as} \quad x,y\to \infty.$$
This heuristic gives an evidence for the following conjecture: \\
\begin{conj}\label{conj} If $L:= \frac{1-\log 2}{\log 2}$, then we have the following dichotomy
\begin{enumerate}
\item: If $u-L\log_2 y\to+\infty$, which implies 
$$y\leq \exp \left\lbrace\frac{\log x}{L\log_2 x}\right\rbrace,$$ 
 Then, we have
$$
A(x,y)\sim \Psi(x,y) \quad \text{as} \quad x,y\to \infty.
$$
\item: If $u- L\log_2 y \to {-\infty}$, which implies that for small $\epsilon>0$
 $$y\geq \exp\left\lbrace\frac{\log x}{(L-\epsilon)\log_2 x}\right\rbrace,$$ 
Then, we have 
$$
A(x,y)=o(\Psi(x,y)) \quad \text{as} \quad x,y\to \infty .
$$
\end{enumerate}  
\end{conj}
Theorem \ref{thm1} and Theorem \ref{thm2} are in the direction of the first case and the second case of Conjecture \eqref{conj} respectively, but the claimed ranges in the conjecture are stronger than the claimed ranges in Theorem \ref{thm1} and Theorem \ref{thm2}, and the reason stems from uniformity assumption about $D(n)$.
\begin{center}Acknowledgement 
\end{center}
I would like to thank Andrew Granville and Dimitris Koukoulopoulos for all their advice and encouragement as well as their valuable comments on the earlier version of the present paper. I am also grateful to Sary Drappeau, Farzad Aryan and Oleksiy Klurman for helpful conversations.  
\section{Preliminaries} \label{Preliminaries}
In this section, we review some results used in the proof of our main theorems. We first fix some notation.
In this chapter $\rho(u)$ is the Dickman-de Bruijn function, as we defined in the introduction. By ~\cite[3.9]{Andrewsmooth} we have the following estimate for $\rho(u)$
 \begin{equation}\label{rho}
 \rho(u) = \left(\frac{e+o(1)}{u\log u}\right)^u \quad \text{as} \quad u\rightarrow \infty.
 \end{equation}
 \bigskip
\begin{theorem}[Hildebrand~\cite{Hil}] The estimate 
\begin{equation}\label{psi}
\Psi(x,y)=x\rho(u)\left(1+O_\epsilon\left(\frac{\log (u+1)}{\log y}\right)\right)
\end{equation}
holds uniformly in the range
\begin{equation}\label{smallu}
x\geq 3,\qquad 1\leq u \leq \frac{\log x}{(\log_2 x)^{\frac{5}{3}+\epsilon}}, \quad \text{that is,} \quad y\geq \exp\left((\log_2 x)^{\frac{5}{3}+\epsilon}\right),
\end{equation}
where $\epsilon$ is any fixed positive number.
\end{theorem}
Combining \eqref{psi} with the asymptotic formula \eqref{rho}, one can arrive at the following simple corollary
\bigskip
\begin{corollary}\label{simplepsi}
We have
$$\Psi(x,y) = xu^{-(u+o(u))},$$
as $y$ and $u$ tend to infinity, uniformly in the range \eqref{smallu}, for any fixed $\epsilon>0$.
\end{corollary}

We will apply this estimate in the proof of Theorem $\ref{thm2}$. However this estimate of $\Psi(x,y)$ is not very sharp for large values of $u$, for which the saddle point method is more effective.\\
Let $\alpha:=\alpha(x,y)$ be a real number satisfying
\begin{equation}\label{alpha}
\sum_{p\leq y} \frac{\log p}{p^{\alpha}-1}=\log x.
\end{equation} 
One can show that $\alpha$ is unique. This function will play an essential role in this work, so we briefly recall some fundamental facts of this function that are used frequently. By~\cite[Lemma 3.1]{TenDe} we have the following estimates for $\alpha$.

\begin{equation}\label{alph-general}
\alpha(x,y)= \frac{\log \left(1+y/\log x\right)}{\log y} \left\{1+O\left(\frac{\log_2 y}{\log y}\right)\right\} \quad \quad x\geq y\geq 2.
\end{equation}
For any $\epsilon>0$, we have the particular cases
\begin{equation}\label{alph}
\alpha(x,y)= 1-\frac{\xi(u)}{\log y} +O\left(\frac{1}{L_\epsilon (y)}+ \frac{1}{u(\log y)^2}\right) \qquad \text{if}\,\,\, y\geq (\log x)^{1+\epsilon},
\end{equation}
where 
\begin{equation}\label{L_epsilon}
L_\epsilon(y)=\exp\left\{(\log y)^{3/5-\epsilon}\right\},
\end{equation}
and $\xi(t)$ is the unique real non-zero root of the equation
\begin{equation}\label{XI}
e^{\xi(t)}= 1+t\xi(t).
\end{equation}
Also for small values of $y$, we have
\begin{equation}\label{alph'}
\alpha(x,y)= \frac{\log (1+\frac{y}{\log x})}{\log y} \left\lbrace1+O\left(\frac{1}{\log y}\right)\right\rbrace \,\,\,\text{if}\,\,\, 2\leq y\leq (\log x)^2.
\end{equation}


We now turn to another ingredient related to the behaviour of $\Psi(x,y)$. The following estimate is a special case of a general result of
de La Breteche and Tenenbaum~\cite[Theorem 2.4]{TenDe}. 
\bigskip
\begin{theorem}\label{locald<y}

If $d\leq y$, then uniformly for $x\geq y \geq 2$ we have
\begin{equation}\label{d<x}
\Psi(x/d,y) =\left\lbrace1+ O\left(\frac{1}{u}+\frac{\log y}{y}\right)\right\rbrace \frac{\Psi(x,y)}{d^\alpha}.
\end{equation}
\end{theorem}
\bigskip


 
 We can deduce the following lemma by Theorem \ref{locald<y} which completes the proof of Theorem \ref{thm3} 
\begin{lemma}\label{smally}
If $y\geq 2$ and $y=o(\log x)$, then we have
\begin{equation}
\Psi(x/y, y)\sim \Psi(x,y) \quad \text{as} \quad x\to \infty.
\end{equation}
\end{lemma}

\begin{proof}
$(i)$: Let $y\geq (\log_2 x)^2$ and $y=o(\log x)$. By applying~\eqref{d<x}, if $d=y$, we obtain
\begin{equation}
\Psi(x/y,y)=\frac{\Psi(x,y)}{ y^\alpha} \left\lbrace 1+O\left(\frac{\log y}{y}\right)\right\rbrace.
\end{equation}
 By combination of the above estimate along with ~\eqref{alph'}, we get
\begin{equation}\label{Psi(x/y,y)}
\Psi(x/y,y) = \frac{\Psi(x,y)}{\left(1+\frac{y}{\log x}\right)^{1+O\left(\frac{1}{\log y}\right)}}\left\lbrace1+O\left( \frac{\log y}{y}\right)\right\rbrace.
\end{equation}
We remark again that $y=o(\log x)$, so we obtain 
$$\frac{1}{\left(1+y/\log x\right)^{1+O\left(1/\log y\right)}} \rightarrow 1\qquad \text{when} \qquad x \rightarrow \infty.$$
Also, we have
$$\frac{\log y}{y}\to 0 \quad \text{when} \quad x\to \infty,$$
since $y\geq (\log_2 x)^2$. Thus, by~\eqref{Psi(x/y,y)}, we conclude 
$$\frac{\Psi(x/y,y)}{\Psi(x,y)}\rightarrow 1 \quad \text{when}\quad x\rightarrow \infty.$$
$(ii):$ Let $2\leq  y \leq (\log_2x)^2$, then by recalling Ennola's theorem \ref{ennola}, we get
\begin{equation}
\begin{split}
\Psi(x/y,y)&= \frac{1}{\pi(y)!} \prod_{p\leq y} \frac{\log x/y}{\log p}\left\lbrace 1+O\left(\frac{y^2}{\log x \log y}\right) \right\rbrace\\
& =\frac{1}{\pi(y)!}\prod_{p\leq y} \frac{\log x}{\log p} \prod_{p\leq y }\left(1-\frac{\log y}{\log x}\right)\left\lbrace 1+O\left(\frac{y^2}{\log x \log y}\right) \right\rbrace\\
& =\Psi(x,y) \left(1+ O\left(\pi(y) \frac{\log y}{\log x}\right)\right)\\
&= \Psi(x,y)\left( 1+O\left(\frac{y}{\log x}\right)\right),
\end{split}
\end{equation}
which gives that
$$\Psi(x/y,y) \sim \Psi(x,y) \quad \text{as} \quad x\to \infty, $$
and this completes the proof.
\end{proof}
\bigskip

Finally, we define
$$
\theta (x,y,z) := \# \{ n\leq x : p|n \Rightarrow z\leq p\leq y \} .
$$
This function has been studied extensively in the literature. Namely   Friedlander~\cite{Fried} and Saias~\cite{II,III} gave several estimates for $\theta(x,y,z)$ in different ranges. The following theorem is due to Saias~\cite[Theorem 5]{III} which is used in Section $4$.
\begin{theorem} 
There exists a constant $c>0$ such that for $x\geq y\geq z\geq 2$ we have

\begin{equation}\label{teta}
\theta(x,y,z) \leq c\frac{\Psi(x,y)}{\log z}.
\end{equation}
\end{theorem}

\section{Proof of Theorem~\ref{thm1}} We begin this section by setting some notation. Let $\eta$ be defined by
$$\eta := \frac{1}{\log_3 y},$$
and set
\begin{equation}\label{N1}
N:=\left\lfloor\frac{\log_2 y -\log\eta}{\log 2}+2\right\rfloor,
\end{equation}
which play an essential role in process of the proof.\\
The idea of the proof of Theorem \ref{thm1} is a combination of some probabilistic and combinatorial techniques. Before going through the details, we give a sketch of proof here.\\
The first step of proving Theorem~\ref{thm1} is to study the number of all prime factors of $n$ in the {\it narrow intervals} 
$$
J_i:=\left[(1-\kappa)y^{1-\frac{1}{2^i}},y^{1-\frac{1}{2^i}}\right], \qquad 1\leq i\leq N,
$$
of {\it multiplicative length} $(1-\kappa)^{-1}$, where $\kappa$ is defined as
\begin{equation}\label{Kappa}
\kappa := \frac{\eta}{2N}.
\end{equation}
Also, we define the {\it tail interval} 
$$J_{\infty}:= [(1-\kappa)y,y].$$
Let $\omega_i(n)$ be the number of prime factors of $n$ in $J_i$ for each $i\in\{1,2,\dots,N,\infty\}$, more formally
\begin{equation}\label{omega_i}
\omega_i(n):= \# \left\lbrace p|n : p\in J_i\right\rbrace.
\end{equation}

We define $\mu_i(x,y)$ to be the expectation of $\omega_i(n)$, defined by
\begin{equation}\label{muidef}
\mu_i(x,y):=\frac{1}{\Psi(x,y)}\sum_{n\in S(x,y)}\omega_i(n),
\end{equation} 
 In Proposition \ref{p1}, we will prove that for almost all $y-$smooth integers the value of $\omega_i(n)$ exceeds $\mu_i(x,y)/2$. We establish this by applying the Chebyshev's inequality
\begin{equation}\label{Chebyshev}
\frac{\# \{n\in S(x,y): \omega_i(n)\leq \mu_i(x,y)/2\}}{\Psi(x,y)}\leq \frac{4\sigma_i^2(x,y)}{\mu_i^2(x,y)},
\end{equation}
where 
\begin{equation}\label{vardef}
\sigma_i^2(x,y):=\frac{1}{\Psi(x,y)}\sum_{n\in S(x,y)}\left(\omega_i(n)-\mu_i(x,y)\right)^2,
\end{equation}
is the variance of $\omega_i(n)$ and $ i\in \{1,2,\dots,N,\infty\}$. We will conclude that there is at least one prime factor $p_i$ in each $J_i$ for $1\leq i\leq N$ and $N$ prime factors $q_1,\dots, q_N$ in $J_\infty$. Then by using the product of these prime factors in Corollary \ref{cor1}, we will find a divisor $D_j$ of $n$ such that 
$$(1-\kappa)^N y^{N-j/2^N} \leq D_j \leq y^{N-j/2^N}, $$ 
for an integer $j$ in $\{0,1,\dots 2^N-1 \}$.\\
Then, we fix an integer $n$ in $S((1-\eta)x,y)$, and by defining $m:= \frac{n}{\prod_{i=1}^Np_i q_i}$, we will easily show that there is a divisor $d_j$ of $n$, such that
$$\frac{\sqrt{n}}{y^N}y^ {j/2^N} <d_j< \frac{\sqrt{n}}{y^N} y^{(j+1)/2^N}.$$
Multiplying $D_j$ and $d_j$ and using the definitions of $\eta$, $\kappa$ and $N$, gives a new divisor $d$ of $n$ that helps us to write $n$ as the product of two divisors less than $\sqrt{x}$.\\\\

Before stating technical lemmas we get an estimate for the expected value of $\omega_i(n)$ for all $1\leq i\leq N$ and $i=\infty$. By changing the order of summation in \eqref{muidef}, we can easily see that 
\begin{equation}\label{Expectation}
\mu_i(x,y)=\sum_{p\in J_i} \frac{\Psi(x/p,y)}{\Psi(x,y)}.
\end{equation}
By \eqref{d<x}, we have the following estimate 

\begin{equation}\label{35}
 \begin{split}
 \mu_i(x,y)&=\sum_{p\in  J_i} \frac{1}{p^{\alpha}}\left(1+O\left(\frac{1}{u}+ \frac{\log y}{y}\right)\right),
 \end{split}
 \end{equation}
 for all $1\leq i \leq N$ and $x\geq y\geq 2$.
 Also, we obtain the following estimate for $\mu_i(x/q,y)$, where $q$ is a prime divisor of $n$.
\begin{equation}\label{mu_i-estimate}
\mu_i(x/q,y)=\sum_{p\in J_i}\frac{1}{p^{\alpha_q}}\left\{1+O\left(\frac{1}{u_q}+ \frac{\log y}{y}\right)\right\},
\end{equation}
where $u_q:= u-\log q/\log y$. By substitution we obtain $x/q =y^{u_q}$. Set the saddle point $\alpha_q:=\alpha (x/q,y)$, defined as the unique real number satisfying in 
\begin{equation}\label{alphaq}
\sum_{p\leq y} \frac{\log p}{p^{\alpha_q}-1}=\log(x/q).
\end{equation}\\\\

We are ready to prove the following lemma that shows the difference between $\mu_i(x/q,y)$ and $\mu_i(x,y)$ is small.  
\begin{lemma}\label{main-lemma} Let $q$ be a prime divisor of $n\in S(x,y)$, then we have
$$\big|\mu_i(x/q,y)-\mu_i(x,y)\big|\ll \frac{\mu_i(x,y)}{u}. $$
\end{lemma}
\begin{proof} 
We use the estimate
\begin{equation}\label{alphahil}
0<-\alpha'(u):= -\frac{d\alpha(u)}{du}\asymp \frac{\bar{u}}{u^2 \log y},
\end{equation}
established in ~\cite[formula 6.6]{Hil&Ten}, where $\bar{u}:= \min \{u, \frac{y}{\log y}\}$.  By \eqref{alphahil}, we deduce \\
\begin{equation}\label{alphahil2}
\big|\alpha^{'}(u)\big| \ll \frac{1}{u\log y}.
\end{equation}
Then applying \eqref{alphahil2}, gives that
\begin{equation}\label{diff alpha}
\begin{split}
\alpha -\alpha_q \leq \int_{u_q}^{u} \big|\alpha^{'}(v)\big| dv &\ll \int_{u}^{u_q} \frac{dv}{v\log y}\\
&= \frac{1}{\log y} \log \left(\frac{u}{u_q}\right) \asymp\frac{\log q}{\log y \log x}.
\end{split}
\end{equation}
By expanding $\mu_i(x/q,y)-\mu_i(x,y)$ and using~\eqref{Expectation} and~\eqref{mu_i-estimate}, we get
\begin{equation}\label{negative}
\begin{split}
\big|\mu_i(x/q,y)-\mu_i(x,y)\big| &= \big|\sum_{p\in J_i} \left(\frac{\Psi(x/pq,y)}{\Psi(x/q,y)}- \frac{\Psi(x/p,y)}{\Psi(x,y)}\right)\big| \\
&\leq \sum_{p\in J_i} \frac{1}{p^\alpha}\left\lbrace\big| p^{\alpha -\alpha_q}-1 \big|+O\left(\frac{1}{u}+ \frac{\log y}{y}\right)\right\rbrace.
\end{split}
\end{equation}
By the Taylor expansion of the exponential function and invoking~\eqref{diff alpha} we obtain
\begin{equation}
\exp\{(\alpha-\alpha_q)\log p\}-1 \ll \frac{\log p \log q}{\log y \log x}.
\end{equation}
We recall that $p,q \leq y$ for $1\leq i\leq N$ and $i=\infty$. From this we infer that 
$$\big| p^{\alpha-\alpha_q}-1\big|\ll \frac{1}{u},$$
this finishes the proof. 
\end{proof}
In the following lemma we shall find an upper bound for $\sigma_i^2(x,y)$ (defined in \eqref{vardef}) for each $i\in \{1,2,\dots, N,\infty\}$.
\begin{lemma}\label{variance}
We have 
$$\sigma_i^2 (x,y) \ll \mu_i(x,y)+ \mu_i^2(x,y)/u,$$
where $i\in \{1,2,\dots ,N,\infty\}.$
\end{lemma}
\begin{proof}
By the definition of $\sigma_i^2 (x,y)$ in \eqref{vardef}, we have
\begin{equation*}
\sigma_i^2(x,y) =\frac{1}{\Psi(x,y)}\sum_{n\in S(x,y)}\left[ \omega_i^2(n)-2\mu_i(x,y)\omega_i(n)+\mu_i^2(x,y)\right].
\end{equation*}
Using the definition of $\omega_i(n)$ in \eqref{omega_i}, gives
\begin{equation*}
\sum_{n\in S(x,y)} \omega_i(n)=\sum_{n\in S(x,y)} \sum_{ p\in J_i} \mathbb{1}_{p|n} = \sum_{p\in J_i}\Psi(x/p,y),
\end{equation*}
where the indicator function $\mathbb{1}_{p|n}$ is $1$ or $0$ according to the prime $p$ divides $n$ or not. By the definition of $\mu_i(x,y)$ in ~\eqref{Expectation}, one can deduce that
$$\sum_{n\in S(x,y)}\omega_i(n)= \Psi(x,y) \mu_i(x,y).$$
 By applying \eqref{Expectation} and the equation above, we obtain

\begin{equation*}
\begin{split}
\Psi(x,y)\sigma_i^2(x,y) &=\sum_{n\in S(x,y)}\left[ \omega_i^2(n)-2\mu_i(x,y)\omega_i(n)+\mu_i^2(x,y)\right] \\
&=\sum_{n\in S(x,y)} \omega_i^2(n)-2\Psi(x,y)\mu_i^2(x,y)+\psi(x,y)\mu_i^2(x,y)\\
& =\left(\sum_{\substack{p,q \in J_j\\ p\not= q}} \Psi(x/pq,y)\right)- \Psi(x,y)\mu_i^2(x,y)
+\sum _{p\in J_i}\Psi(x/p,y)\\
&:=S_1 +S_2,
\end{split}
\end{equation*}
where $S_1:=\sum_{\substack{p,q \in J_j\\ p\not= q}} \Psi(x/pq,y)- \Psi(x,y)\mu_i^2(x,y) $ and $S_2:= \sum _{p\in J_i}\Psi(x/p,y)$.
 We next find an upper bound for each $S_i$. We first consider $S_1$, by using ~\eqref{Expectation} we can get
\begin{equation}\label{S1}
\sum_{\substack{p,q \in J_i\\ p\not= q}} \Psi(x/pq,y)-
\Psi(x,y)\mu_i^2(x,y) \leq \sum_{p\in J_i}\Psi(x/p,y)\left(\mu_i(x/p,y)-\mu_i(x,y)\right).
\end{equation} 
By Lemma~\ref{main-lemma} and using~\eqref{S1}, we obtain the following upper bound for $S_1$
\begin{equation}\label{upper-S_1}
S_1\leq C \frac{\Psi(x,y)\mu_i^2(x,y)}{u},
\end{equation}
where $C$ is a positive constant. It remains to estimate $S_2$, from \eqref{Expectation} we have
$$S_2= \Psi(x,y)\mu_i(x,y).$$
By substituting the upper bounds for $S_1$ and $S_2$,  we get
$$\sigma_i^2(x,y) = \frac{S_1+S_2}{\Psi(x,y)} \ll\left( \mu_i(x,y)+\frac{\mu_i^2(x,y)}{u}\right),$$
and the proof is complete.
\end{proof}

Now we give an order of magnitude for $\mu_i(x,y)$, where $i\in \{1,2,\dots, N, \infty\}$
\begin{lemma}\label{Mumagnitude}
We have
$$\mu_i(x,y) \asymp \kappa \frac{Y^{1-\frac{1}{2^i}}}{\log y},$$
where  $i\in \{1,2,\dots, N,\infty\}$, and $$Y:= y^{1-\alpha}.$$
\end{lemma}
\begin{proof}
By the definition of each $J_i$, we obtain the following simple inequalities
 \begin{equation} \label{low}
\begin{split}
\frac{1}{y^{\alpha(1-1/2^i)}} \# \left\lbrace p\in J_i\right\rbrace\leq\sum_{p\in J_i} \frac{1}{p^{\alpha}} \leq \frac{1}{(1-\kappa)y^{\alpha(1-1/2^i)}} \# \left\lbrace p\in J_i\right\rbrace.
\end{split}
\end{equation} 
 By applying the prime number theorem, we obtain
\begin{equation}\label{pnt}
\begin{split}
\#\{p: p \in J_i\}&=\pi(y^{1-1/2^i})-\pi((1-\kappa)y^{1-1/2^i})\\&=\frac{y^{1-1/2^i}}{\log\left(y^{1-1/2^i}\right)}-\frac{(1-\kappa)y^{1-1/2^i}}{\log\left((1-\kappa)y^{1-1/2^i}\right)}+O\left(\frac{y^{1-1/2^i}}{\log^2 y}\right) \\
&= \frac{y^{1-1/2^i}}{(1-1/2^i) \log y}-\frac{(1-\kappa)y^{1-1/2^i}}{(1-1/2^i)\log y}\left(1+O\left( \frac{\log (1-\kappa)}{\log y}\right)\right)\\
&= \frac{\kappa y^{1-1/2^i}}{(1-1/2^i)\log y}(1+o(1)),
\end{split}
\end{equation}
The last equality is true, since the given values of $\kappa$ and $N$ in \eqref{Kappa} and \eqref{N1} imply
\begin{equation}\label{Kappa-lower}
\kappa \asymp 1/(\log_2 y \log_3 y).
\end{equation}
By substituting \eqref{pnt} in \eqref{low} we have
\begin{equation}\label{Mu1}
\mu_i(x,y)\asymp\kappa\frac{ Y^{1-1/2^i}}{\log y},
\end{equation}
\end{proof}
\bigskip
By the above lemmas, we are now ready for proving the following proposition.
\begin{proposition}\label{p1} If $u$ and $y$ satisfy in range given in\eqref{range1}, we have
$$\#\left\{n\in S(x,y): \omega_i(n)> \frac{\mu_i(x,y)}{2}\quad \forall i\in\{1,\dots,N,\infty\}\right\}\sim \Psi(x,y) \quad \text{as} \quad x,y\to \infty,$$
\end {proposition}

\begin{proof} By the Chebyshev's inequality in \eqref{Chebyshev} and using the upper bound for $\sigma_i^2(x,y)$ in lemma \eqref{variance}, we get
 $$\#\left\{n\in S(x,y): \omega_i(n)\le \frac{\mu_i(x,y)}{2}\right\} \ll \Psi(x,y)\left(\frac{1}{\mu_i(x,y)}+ \frac{1}{u}\right).$$
 By the above inequality, we obtain an upper bound for the following set
 \begin{equation} \label{Mi1}
 \begin{split}
M &:= \#\left\{n\in S(x,y): \exists i \in \{1,\dots,N,\infty\}\quad \text{such that} \,\,\, \omega_i(n)\le\frac{\mu_i(x,y)}{2}\right\} \\&\ll \Psi(x,y)\left[\frac{1}{\mu_\infty(x,y)}+\frac{N}{u}+\sum_{i=1}^N\frac{1}{\mu_i(x,y)}\right].\\
 \end{split}
 \end{equation}
Our main task that finishes the proof is to find a range such that $M/\Psi(x,y)$ tends to $0$.\\
By using Lemma \ref{Mumagnitude} and substituting the order of magnitude of $\mu_i(x,y)$ in~\eqref{Mi1}, we get
\begin{equation}\label{Mlow}
M\ll \Psi(x,y) \left[\frac{\log y}{ \kappa Y}+ \frac{N}{u}+ \frac{\log y}{\kappa}\sum_{i=1}^{N} \frac{1}{Y^{1-1/2^i}} \right].
\end{equation}
In what follows, we find a lower bound for $Y$ in two different ranges of $y$ \\ \\
$\mathbf{ (i)}:$ If $y\leq (\log x)^2$, then by \eqref{alph'} $\alpha \leq 1/2+ o(1)$ as $y\to \infty$. Therefore,

$$Y\geq y^{1/2-o(1) } \geq y^{1/3}.$$
By substituting this lower bound in \eqref{Mlow} and using the precise value of $N$ in \eqref{N1}, we have

\begin{equation}
\begin{split}
M& \ll \Psi(x,y) \left[\frac{\log y}{\kappa y^{1/3}}+ \frac{N}{u}+\frac{\log y}{\kappa y^{1/3}} \sum_{i=1}^{N}y^{1/3(2^{i})}\right]\\
& \ll \Psi(x,y)\left[\frac{\log_2 y}{u}+\frac{ y^{1/6}\log y}{\kappa y^{1/3}} \left(1+ O\left(Ny^{-1/12}\right)\right)\right]\\
& \ll \Psi(x,y)\frac{\log y}{\kappa y^{1/6}},
\end{split}
\end{equation}
By using the asymptotic value of $\kappa$ in \eqref{Kappa-lower}, we obtain
$$M \ll \Psi(x,y) \frac{\log y \log_2 y \log_3 y}{y^{1/6}},$$
and clearly we have
 $$M=o(\Psi(x,y)) \quad \text{as} \quad x,y\to \infty,$$
this finishes the proof for the case $y \leq (\log x)^2$.\\\\
$\mathbf{(ii)}:$ If $y\geq (\log x)^2$, by applying~\eqref{alph}, we have
\begin{equation}\label{alpha aproxx}
1-\alpha = \frac{\xi(u)}{\log y} +O\left(\frac{1}{L_\epsilon (y)}+ \frac{1}{u(\log y)^2}\right).
\end{equation}
Using ~\cite [Lemma 8.1]{Tenenbaum}, we have the following estimate of $\xi$

$$\xi(t) =\log (t\log t)+ O\left(\frac{\log_2 t}{\log t}\right) \quad \text{if} \quad t>3.$$
Therefore,
$$1-\alpha = \frac{\log (u\log u)}{\log y}+ O\left(\frac{\log_2 u}{\log y \log u}\right),$$
Thus, we get
\begin{equation}\label{Ybig}
\begin{split}
Y &= u \log u \left[1+O\left(\frac{\log_2 u}{\log u}\right)\right]\\
&\asymp u \log u.
\end{split}
\end{equation}
By combining the above with the estimate in \eqref{Ybig}, and using the value of $N$ in \eqref{N1}, we get 

\begin{equation}
\begin{split}
M &\ll  \Psi(x,y) \left[ \frac{\log y}{\kappa u \log u}+ \frac{N}{u}+ \frac{\log y}{\kappa u \log u} \sum_{i=1}^{N} (u \log u)^{1/2^i}\right]\\
& \ll \Psi(x,y) \left[ \frac{N}{u}+ \frac{\log y}{\kappa u \log u}\left((u\log u)^{1/2} + (u\log u )^{1/2^2}+...+(u\log u)^{1/2^N}\right)\right]\\
&\ll \Psi(x,y)\left[\frac{N}{u}+ \frac{\log y}{\kappa(u\log u)^{1/2}}\left(1+O\left( N(u\log u)^{-1/4}\right)\right)\right]\\
&\ll\Psi(x,y) \left[ \frac{\log_2 y}{u}+ \frac{\log y}{\kappa(u\log u)^{1/2}}\right],
\end{split}
\end{equation}
By using the order of $\kappa$ in \eqref{Kappa-lower}, one can arrive at the following upper bound of $M$
\begin{equation}
M\ll \Psi(x,y)\frac{ \log y \log_2 y \log_3 y}{(u \log u)^{1/2}}.
\end{equation}
So there exists a constant $c$ such that for all $i\in \{1,\dots, N, \infty\}$, we have
\begin{equation}
\begin{split}
\# \left\lbrace n\in S(x,y) : \omega_i(n)> \mu_i (x,y)/2 \quad \forall i \right\rbrace \geq \Psi(x,y)\left(1-c\frac{\log y \log_2 y \log_3 y}{ (u \log u)^{1/2}}\right),
\end{split}
\end{equation} 
and this finishes the proof by letting $$\frac{u\log u}{(\log y \log_2 y \log_3 y)^2 } \to \infty.$$
 \end{proof}
\bigskip
\begin{corollary}\label{cor1}
If $x$ and $y$ satisfy the range \eqref{range1}, then almost all $n$ in $S(x,y)$ are divisible by at least one prime factor $p_i$ in $J_i$, and $N$ prime factors $q_1,..., q_N$ in $J_\infty$. Moreover, the product $\prod_{i=1}^{N}p_i q_i$ has a divisor $D_j$ in each of intervals $[(1-\kappa)^N y^{N-j/2^N},y^{N-j/2^N}]$, where $j\in \{0,1,...,2^N-1\}$.
 \end{corollary}
 
 \begin{proof}
The first part of Corollary is a direct conclusion of Proposition \ref{p1}.\\
 For the second part, let $n$ be a $y-$smooth integer satisfying the first part of Corollary. We fix the following divisor of $n$
 $$D:=\prod_{i=1}^{N} p_i q_i,$$
where $p_i \in J_i$ and $q_1,...,q_N \in J_\infty$.\\
Let $j$ be an arbitrary integer in $\{0,1,...,2^N-1\}$. Moreover, we define 
$$a_0:= N-\sum_{i=1}^{N}a_i,$$
where $a_i$'s get the values $0$ or $1$ such that 
\begin{equation}\label{ai}
\sum_{i=1}^{N}\frac{a_i}{2^i} = j/2^N.
\end{equation}
 We now define the divisor of $D_j$ of $D$ with the following form
  $$D_j:= \prod_{i=1}^N p_i^{a_i}\prod_{i=1}^{a_0} q_i,$$
By using the bounds of  $p_i$s and $q_i$s, one can get the following bounds for $D_j$.
$$ (1-\kappa)^N y^{N-\sum_{i=1}^Na_i/2^i}\leq D_j \leq y^{N-\sum_{i=1}^Na_i/2^i},$$ 
By using \eqref{ai}, we have
$$(1-\kappa)^N y^{N-j/2^N}\leq D_j \leq  y^{N-j/2^N},$$
and this finishes our proof.
\end{proof}
\bigskip
 We are ready now to prove Theorem \ref{thm1}.
 \begin{proof} [\textbf{Proof of Theorem \ref{thm1}}]
 Let $n\leq (1-\eta)x $ be a $y-$smooth integer with at least one prime factor $p_i$ in each  $J_i$ , where $i=1,..,N$, and $N$ prime divisors $q_1,q_2, ...,q_N$ in $ J_\infty$. Set
 $$m:= \frac{n}{\prod_{i=1}^{N} p_i q_i}.$$ 
 By this definition, we get
  $$\frac{n}{\prod_{i=1}^N p_iq_i}\geq \frac{n}{y^{2N}} >\sqrt{n},$$  
 when $4N\leq u$. Thus,
 $$m>\sqrt{n}.$$
 Let $\{r_v\}$ be the increasing sequence of prime factors of $m$ and set $d_v=r_1...r_v$.\\
Clearly, $m$ has at least one divisor bigger than $\frac{\sqrt{n}}{y^N}$. We suppose  that $l$ is the smallest integer such that $d_l \geq \frac{\sqrt{n}}{y^N}$, and evidently we have  $$d_{l-1} \leq \frac{\sqrt{n}}{y^N},$$ 
So, we arrive at the following bounds for $d_l$

\begin{equation}\label{d_i}
\frac{\sqrt{n}}{y^N}\leq d_l \leq yd_{l-1}\leq \frac{\sqrt{n}}{y^{N-1}},
\end{equation}
We pick $k \in \{0,1,2...,2^N-1\}$ such that 
\begin{equation} \label{dk}
\frac{\sqrt{n}}{y^N} y^{k/2^N} \leq d_l \leq \frac{\sqrt{n}}{y^N}y^{(k+1)/2^N}.
\end{equation}
By the second part of Corollary \ref{cor1}, for every $k$ in $\{0,1,...,2^N-1\}$ there exists a divisor $D_k$ such that
 \begin{equation*}\label{dj}
 (1-\kappa)^Ny^{N-k/2^N}\leq D_k \leq y^{N- k/2^N},
 \end{equation*}
 We define $d:= d_l D_k$, we have
\begin{equation*}\label{d}
(1-\kappa)^{N} \sqrt{n}\leq d \leq y^{1/2^N}\sqrt{n},
\end{equation*}
 By using the values of $N$ in \eqref{N1} and $\kappa$ in \eqref{Kappa}, we have
\begin{equation*}\label{y2N}
e^{-\eta/2}\sqrt{n}\leq d \leq e^{\eta/2}\sqrt{n}.
\end{equation*}
Applying the Taylor expansion for exponential functions, gives 
\begin{equation}\label{dbound}
\left(1-\eta +\frac{\eta^2}{2}+O(\eta^3)\right)^{1/2}\sqrt{n} \leq d \leq \left(1+\eta +\frac{\eta^2}{2}+O(\eta^3)\right)^{1/2} \sqrt{n}.
\end{equation}
By using the assumption $n\leq (1-\eta)x$ in the upper bound and lower bound above, we obtain 
$$d\leq \left(1-\frac{\eta^2}{2}+O(\eta^3)\right)^{1/2}\sqrt{x} \leq \sqrt{x},$$
 and 
$$\frac{n}{d} \leq \left(1+\eta+\frac{\eta^2}{2}+ O(\eta^3)\right)^{1/2} \sqrt{n} \leq \left(1-\frac{\eta^2}{2}+O(\eta^3)\right)^{1/2}\sqrt{x} \leq \sqrt{x}.$$
Thus, we can write $n\in S((1-\eta)x,y)$ as the product of two divisors less than $\sqrt{x}$, and we can deduce that 
$$ \Psi\left((1-\eta)x,y\right)\leq A(x,y) \leq \Psi(x,y),$$
By using \eqref{d<x}, we have

$$\frac{\Psi\left((1-\eta)x,y\right)}{\Psi(x,y)} = \left(1-\eta\right)^{\alpha} \left\{ 1+O\left(\frac{1}{u} +\frac{\log y}{y}\right)\right\} \rightarrow 1 \quad \text{as} \quad x,y\to \infty,$$
this finishes the proof.
 
\end{proof}
 
 
 \section{Proof of Theorem\ref{thm2}}

 In this section, we shall study the behaviour  of $A(x,y)$ for large values of $y$. When $y$ takes values very close to $x$, then the set of $y-$smooth integers contains integers having large prime factors. As we explained in the heuristic argument, one can expect that $A(x,y)= o(\Psi(x,y))$. To show this assertion, we recall the idea of Erd\H{o}s used to prove the multiplication table problem for integers up to $x$.\\
We start our argument by giving an upper bound for $A^*(x)$, defined by
\begin{equation}
A^*(x):= \# \left\lbrace ab: a,b\leq \sqrt{x} \quad \text{and} \quad (a,b)=1\right\rbrace.
\end{equation}
We shall find an upper bound of $A^*(x)$ by considering the number of prime factors of $a$ and $b$. We first define 
$$\pi_k(x) := \# \{n\leq x: \omega(n)= k\}$$
Therefore,
\begin{equation}\label{erdos idea}
\begin{split}
A^*(x)&\leq \sum_{k} \min \left\lbrace\pi_k(x) , \sum_{j=1}^{k-1} \pi_{j}(\sqrt{x}) \pi_{k-j} (\sqrt{x})\right\rbrace\\
& \leq \sum_{k}\min \left\lbrace\frac{cx}{\log x} \frac{(\log_2x)^{k-1}}{(k-1)!}, \sum_{j=1}^{k-1}\frac{c\sqrt{x}}{\log \sqrt{x}}\frac{(\log_2 \sqrt{x})^{j-1}}{(j-1)!} \frac{c\sqrt{x}}{\log \sqrt{x}} \frac{(\log_2 \sqrt{x})^{k-j-1}}{(k-j-1)!}\right\rbrace,
\end{split}
\end{equation}
where in the last inequality, we used the well-known result of Hardy and Ramanujan that states there are absolute constants $C$ and $c$ such that
\begin{equation}\label{hardy}
\pi_k(x) \leq \frac{cx}{\log x} \frac{(\log_2x+C)^{k-1}}{(k-1)!} \quad \text{for} \quad k=0,1,2,.. \quad \text{and} \quad x\geq 2.
\end{equation}
By simplifying the upper bound in \eqref{erdos idea} and using Stirling's formula
$$n!\sim n^{n+\frac{1}{2}}e^{-n}$$
we obtain
\begin{equation}\label{A*}
\begin{split}
A^*(x)&\leq \sum_{k} \min \left\lbrace\frac{cx}{\log x}\frac{(\log_2 x)^{k-1}}{(k-1)!} , \frac{4c^2 x}{(\log x)^2} \sum_{j=0}^{k-2} \frac{1}{(k-2)!}{k-2 \choose j}(\log_2 \sqrt{x})^{k-2}\right\rbrace\\
&= \sum_{k} \min \left\lbrace\frac{cx}{\log x}\frac{(\log_2 x)^{k-1}}{(k-1)!} ,\frac{4c^2 x}{(\log x)^2} \frac{(2\log_2 \sqrt{x})^{k-2}}{(k-2)!} \right\rbrace \\
&= \sum_{k\leq \frac{\log_2x}{\log 2}} \frac{4c^2 x}{(\log x)^2} \frac{(2\log_2\sqrt{x})^{k-2}}{(k-2)!}  
+  \sum_{k> \frac{\log_2 x}{\log 2}} \frac{cx}{\log x}\frac{(\log_2 x)^{k-1}}{(k-1)!}  \\
&\ll \frac{x}{(\log x)^{1-\frac{1+\log \log 2}{\log 2}}(\log_2x)^{1/2}} \to 0 \quad \text{as} \quad x\to \infty.
\end{split}
\end{equation}

 We shall get the same upper bound for $A(x)$. Let $n\leq x$ and there are $a$ and $b$ less than $\sqrt{x}$ such that $n=ab$. If $(a,b)= 1$ then $n$ is counted by $A(x)$ , and if $(a,b)=d>1$ then we can write $n$ as $n= a^{'} b^{'}d^2$ such that $(a^{'},b^{'})=1$. So, $\frac{n}{d^2}\leq \frac{x}{d^2}$, and $\frac{n}{d^2}$ will be counted by $A(\frac{x}{d^2})$. Therefore,
$$A(x) \leq \sum_{d \leq {\sqrt{x}}}A^*({\frac{x}{d^2}}) \ll A^*(x)$$
By \eqref{A*}, we get
$$A(x) \ll \frac{x}{(\log x)^{1-\frac{1+\log \log 2}{\log 2}}(\log_2x)^{1/2}}. $$
Thus,
$$A(x)=o(x) \quad \text{as} \quad x\rightarrow \infty.$$\\\\

 Motivated by Erd\H{o}s' idea for the multiplication table of integers up to $x$, we apply a similar method to find an upper bound for $A(x,y)$.\\
 The first step of proof is to study the following function which plays a crucial role in this section. Let
$$N_{k}(x,y,z) :=\#\{n\in S(x,y): \Omega_{z}(n)=k\},$$
where $\Omega_{z}(n)$ is the truncated version of $\Omega(n)$, only counting divisibility by primes not exceeding $z$ with their multiplicities. In other words
$$\Omega_{z}(n):= \sum _{\substack{p^{v}||n\\p\leq z}}v.$$
In the following lemma, by using induction on $k$, we shall find an upper bound of type \eqref{hardy} for $N_k(x,y,z)$. The reason of applying truncation is to sieve out prime factors exceeding some power of $y$ which are the cause of big error terms as $k$ increases in each step of induction. The upper bound of $N_k(x,y,z)$ leads us to generalize Erd\H{o}s' idea for $y-$smooth integers in a certain range of $y$.\\\\
\begin{lemma}\label{omegaz}
Let $u \leq (C-\epsilon)\log \log y$, where $C$ is a positive constant and $\epsilon>0$ is arbitrarily small. Set the parameter $z$ such that
$$\log \log z\ll u.$$
Then, there are constants $A$ and $B$ such that the inequality
\begin{equation}\label{Omegaz}
N_{k}(x,y,z)\leq \frac{A\Psi(x,y)}{\log z}\frac{(\log \log z+B)^k}{k!}
\end{equation}
holds for every integer $k>0$.
\end{lemma}
\begin{proof}
When $k=0$, by \eqref{teta}, evidently we have
$$N_0 (x,y,z) = \theta(x,y,z) \leq c\frac{\Psi(x,y)}{\log z},$$
where $c>0$ is a constant.
When $k=1$, we can represent $n$ as $n=pm$, where $p\leq z$ and every prime factor $q$ of $m$ is between $z$ and $y$, then using the definition of $\theta(x,y,z)$ we have

$$N_1(x,y,z)=\sum_{p\leq z}\sum_{\substack{m\leq x/p\\ q|m \Rightarrow z\leq q\leq y}}1=\sum_{p\leq z} \theta(x/p,y,z). $$
By applying the estimate \eqref{d<x} and \eqref{teta}, there is constant $c$ such that

$$N_1(x,y,z) \leq\sum_{p\leq z} \frac{c\Psi(x/p,y)}{\log z} = c\frac{\Psi(x,y)}{\log z}\sum_{p\leq z}\frac{1}{p^{\alpha}}\left\lbrace1+O\left(\frac{1}{u}\right)\right\rbrace.$$
 For the last summand we have

\begin{equation}\label{sumz}
\begin{split}
\sum_{p\leq z}\frac{1}{p^\alpha}&= \sum_{p\leq z} \frac{1}{p} \left(p^{1-\alpha}\right)\\
&= \sum_{p\leq z} \frac{1}{p} \left\lbrace1+O\left((1-\alpha)\log p\right)\right\rbrace,
\end{split}
\end{equation}
since  $(1-\alpha)\log p \leq (1-\alpha)\log z$, and $(1-\alpha)\log z$ is bounded in our range (see \eqref{1-alpha}). Therefore,
\begin{equation}
\sum_{p\leq z}\frac{1}{p^{\alpha}}= \log_2 z + O\left((1-\alpha) \log z\right),\\
\end{equation}
By using the estimate of $\alpha$ in \eqref{alph} and the upper bound of $z$, we get 
\begin{equation}\label{1-alpha}
(1-\alpha)\log z\ll \frac{\log u}{\log y} \log z \ll \frac{\log u}{\log_2 y} \ll \frac{\log_3 y}{\log_2 y},
\end{equation}
and we obtain
\begin{equation}\label{sumlogz}
\sum_{p\leq z} \frac{1}{p^{\alpha}}= \log_2 z +O\left(\frac{\log _3 y}{\log _2 y}\right).
\end{equation}
Thus,
\begin{equation}\label{O}
\sum_{p\leq z}\frac{1}{p^{\alpha}}\left\lbrace 1+O\left(\frac{1}{u}\right)\right\rbrace = \log \log z+ O(1),
\end{equation}
since we have $\log \log  z \ll u$. \\
Substituting \eqref{O} in the upper bound of $N_1(x,y,z)$, gives 
$$N_1(x,y,z) \leq \frac{c\Psi(x,y)}{\log z} \left(\log_2z+O(1)\right).$$
 We will show the lemma with $A=c$ and $B=O(1)$.
 We argue by induction: we assume that the estimate in \eqref{Omegaz} is true for any positive integer $k$, we now prove it for $n\in S(x,y)$ with $\Omega_z(n)=k+1$. There are $k+1$ ways to write $n$ as $n=pm_1m_2$ such that $p\leq z$ and $\Omega_z(m_1)=k$ and every prime factor of $m_2$ is greater than $z$. Then we have
\begin{equation*}
\begin{split}
N_{k+1}(x,y,z)&= \frac{1}{(k+1)} \sum_{p\leq z}\sum_{\substack{m_1\in S(x/(p),y)\\ \Omega_{z}(m_1)=k \\ m_2 \in S(x/(pm_1),y)\\q|m_2 \Rightarrow q>z}} 1
 \leq \frac{1}{(k+1)} \sum_{p\leq z}\sum_{\substack{m_1\in S(x/(p),y)\\ \Omega_{z}(m_1)=k}} 1\\
&= \frac{1}{(k+1)} \sum_{p\leq z}N_{k}(x/p , y, z)
\end{split}
\end{equation*}
By the assumption for $\Omega_{z}(n)=k$ and \eqref{d<x}, we get

\begin{equation}\label{k+1}
\begin{split}
 N_{k+1}(x,y,z)&\leq \frac{A(\log_2z+B)^k}{\log z(k+1)!} \sum_{p\leq z} \Psi(x/p,y) \\
& =\frac{A\Psi(x,y)}{\log z} \frac{(\log_2z+B)^k}{(k+1)!}\sum_{p\leq z}\frac{1}{p^{\alpha}}\left\lbrace1+O\left(\frac{1}{u}\right)\right\rbrace.
\end{split}
\end{equation}
 By applying the estimate in \eqref{O}, we arrive at the following bound for $N_{k+1}(x,y,z)$

$$N_{k+1}(x,y,z)\leq \frac{A\Psi(x,y)}{\log z} \frac{(\log_2z+B)^{k+1}}{(k+1)!},  $$
 so we derived our desired result.
\end{proof}
\bigskip






\begin{proof}[\textbf{Proof of Theorem \ref{thm2}}]
For a small $\epsilon>0$, we set $u < \left(\frac{\lambda}{\log 2}-\epsilon\right)\log_2 y$, where $\lambda$ is a fixed real number in the open interval $(1-2\log 2 , 1-\log 2)$.\\  We now set $z$ satisfying
\begin{equation}\label{lambda}
\log \log z= \frac{\log 2}{\lambda} u,
\end{equation}
 so the given ranges of $u$ and $z$ satisfy the conditions of Lemma \ref{omegaz}.\\
By the definition of $A(x,y)$, we have the following evident bound of $A(x,y)$
\begin{equation}\label{min}
\begin{split}
A(x,y) \leq \sum_{k} \min \left\lbrace\sum_{\substack{n\in S(x,y)\\ \Omega_{z}(n)=k}}1 , \sum_{j=1}^{k-1} \sum_{\substack{a\in S(\sqrt{x},y)\\ \Omega_{z}(a)=j}}1\sum _{\substack{b\in S(\sqrt{x},y)\\\Omega_{z}(b)=k-j}}1\right\rbrace.
\end{split}
\end{equation}
 We set  
$$L=\lfloor H \log_2 z\rfloor,$$
where 
$$H:= \frac{1-\lambda}{\log 2}.$$
We have $1-2\log 2 < \lambda< 1-\log 2$. Thus, $1<H<2$.\\\\
By using \eqref{min}, we write the following bound for $A(x,y)$
\begin{equation}
\begin{split}
A(x,y) &\leq \# \left\lbrace n\in S(x,y): \Omega_z(n)> L\right\rbrace +\# \left\lbrace ab: a,b \in S(\sqrt{x},y), \Omega_z(a)+\Omega_z(b)\leq L\right\rbrace\\
&= \sum_{k>L}N_{k}(x,y,z) + \sum_ {k\leq L}\sum_{j=0}^{k} N_{j}(\sqrt{x},y,z)N_{k-j}(\sqrt{x},y,z).
\end{split}
\end{equation}
By applying Lemma \ref{omegaz}, we have 
\begin{equation}
\begin{split}
A(x,y) &\ll \sum_{k> L}\frac{\Psi(x,y)}{\log z} \frac{(\log_2z+c)^k}{k!}+ \sum_{k\leq L}\sum_{j=0}^{k}\frac{\Psi^2(\sqrt{x},y)}{\log ^2 z} \frac{(\log_2 z +c)^j}{j!} \frac{(\log_2 z+c)^{k-j}}{(k-j)!}\\
& =\sum_{k>L} \frac{\Psi(x,y)}{\log z} \frac{(\log_2 z+c)^k}{k!}+ \sum_{k\leq L}\frac{\Psi^2(\sqrt{x},y)}{\log ^2 z} \sum_{j=0}^{k} \frac{1}{k!}{k\choose j}(\log_2 z+ c)^k\\
&=\sum_{k>L}\frac{\Psi(x,y)}{\log z} \frac{(\log_2 z+c)^k}{k!}+\sum_{k\leq L} \frac{\Psi^2(\sqrt{x},y)}{\log ^2 z} \frac{(2\log_2 z+c)^k}{k!}.\\
\end{split}
\end{equation}
By applying the simple form of $\Psi(x,y)$ in Corollary \ref{simplepsi}, and using the assumption \eqref{lambda}, we get
\begin{equation} \label{gran}
 \frac{\Psi^2(\sqrt{x},y)}{\Psi(x,y)}\asymp (\log z)^\lambda \quad \text{as} \quad u,y\to \infty.
 \end{equation}
 Thus,
\begin{equation}\label{A**}
A(x,y)\ll \frac{\Psi(x,y)}{\log z} \sum_{k>L}\frac{(\log_2 z+c)^k}{k!}+ \frac{(\log z)^{\lambda}\Psi(x,y)}{\log ^2 z} \sum_{k\leq L}\frac{(2\log_2 z+c)^k}{k!}.
\end{equation}
The maximum values of functions in the above summands (with respect to $k$) are attained at $k=\lfloor\log_2 z\rfloor$ and $k=\lfloor2\log_2 z\rfloor$ respectively. We have $\log \log z < L < 2 \log \log z$, so the function in the first summation in \eqref{A**} in decreasing for $k>L$, and by using Stirling's formula $k!\sim k^{k+\frac{1}{2}}e^{-k}$, we have

\begin{equation}\label{sum1}
\begin{split}
\sum_{k> L} \frac{(\log_2z)^k}{k!}&= \sum_{H\log_2 z<k\leq e\log_2 z} \frac{(\log_2 z)^k}{k!}+\sum_{e\log_2z < k \leq 2e\log_2z} \frac{(\log_2z)^k}{k!}+\sum_{k> 2e\log_2 z} \frac{(\log_2 z)^k}{k!}\\
& \ll (\log_2z)\left(\left(\frac{e}{H}\right)^{H\log_2z} +1\right)\\& \ll \frac{1}{(\log z)^{H\log H-H}}.
\end{split}
\end{equation}
The function in the second summation in \eqref{A**} is increasing for $k\leq L$, and we have
\begin{equation}\label{sum2}
\begin{split}
\sum_{k\leq L}\frac{(2\log_2 z+c)^k}{k!} \ll (\log_2 z)\left(\frac{2e}{H}\right)^{H\log_2z}= \frac{1}{(\log z)^{H\log H -H -H\log 2}}
\end{split}
\end{equation}
 Substituting the upper bounds obtained in \eqref{sum1} and \eqref{sum2} in  \eqref{A**}, and using the definition of $H$, gives
$$A(x,y) \ll \frac{\Psi(x,y)}{(\log z)^ {G(H)}},$$
where 
$$G(H) := 1+H\log H -H.$$
The function $G(H)$ is an increasing function in the interval $(1,2)$ with a zero at $H=1$. Thus, for any arbitrary $1-2\log 2 < \lambda < 1-\log 2$, we have
 $$A(x,y)= o(\Psi(x,y)) \quad \text{as} \quad x,y\to \infty,$$ 
 so we obtained our desired result.
\end{proof}

\bibliographystyle{plain}
\bibliography{references}

\end{document}